\documentclass[11pt]{amsart}
\usepackage{amsthm}
\usepackage{array}
\usepackage{amsmath}
\usepackage{enumerate}
\usepackage{tikz}
\usetikzlibrary{calc}
\usepackage{amssymb}
\usepackage{latexsym}
\usepackage{amsfonts}
\usepackage{color}
\usepackage{mathrsfs}
\usepackage{epsfig,helvet}

\theoremstyle{plain} \numberwithin{equation}{section}
\newtheorem{thm}{Theorem}[section]
\newtheorem{theorem}[thm]{Theorem}
\newtheorem{lemma}[thm]{Lemma}
\newtheorem{corollary}[thm]{Corollary}

\newtheorem{definition}[thm]{Definition}
\newtheorem{proposition}[thm]{Proposition}
\begin{document}
\setcounter{page}{1}

\title[ $2$-nilpotent multiplier ]{On $2$-Nilpotent Multiplier of Lie Superalgebras  }

\author[Padhan]{Rudra Narayan Padhan}
\address{Department of Mathematics, National Institute of Technology,  \\
         Rourkela, 
          Odisha-769028 \\
               India}
\email{rudra.padhan6@gmail.com}

\author[Nandi]{Nupur Nandi}
\address{Department of Mathematics, National Institute of Technology,  \\
         Rourkela, 
          Odisha-769028 \\
               India}
\email{nupurnandi999@gmail.com}

\author[Pati]{K.C. Pati}
\address{Department of Mathematics, National Institute of Technology,  \\
         Rourkela, 
          Odisha-769028 \\
                India}
\email{kcpati@nitrkl.ac.in }

\maketitle
\section{Abstract}
In this article we define the $c$-nilpotent multiplier of a finite dimensional Lie suepralgebra. We characterize the structure of $2$-nilpotent multiplier of finite dimensional nilpotent Lie superalgebras whose derived subalgebras have dimension at most one. Then we give an upper bound on the dimension of $2$-nilpotent multiplier of any finite dimensional nilpotent Lie superalgebra. Moreover, we discus the $2$-capability of special as well as odd Heisenberg Lie superalgebras and abelian Lie superalgebras.

\section{Introduction}

Mathematics research comprises mainly classification of algebraic objects. Recently, several authors have considered $c$-nilpotent multiplier of Lie algebra. Salemkar et al. \cite{salemkar} developed a theory of $c$-nilpotent multiplier of Lie algebras. Let $0\rightarrow R \rightarrow F \rightarrow L \rightarrow 0$ be a free presentation of a finite dimensional Lie algebra $L$ where $F$ is a free Lie algebra and $R$ is an ideal of $F$. Then for $c\geq 1$, the $c$-nilpotent multiplier of $L$ is defined as $$ \mathcal{M}^{(c)}(L) =\frac{R\cap \gamma_{c+1}(F)}{\gamma_{c+1}(R,F)} $$ 
where $\gamma_{c+1}(F)$ is the $(c+1)$-th term of the lower central series of $F$ and $\gamma_{1}(R,F)=R$, $\gamma_{c+1}(R,F)=[\gamma_{c}(R,F),F]$. This definition is analogous to the definition of the Baer-invariant of a group with respect to the variety of nilpotent groups of class at most $c$ \cite{baer}. For $c=1$, the Schur multiplier $\mathcal{M}^{(1)}(L)=\mathcal{M}(L)$ has been studied by Batten \cite{batten}. 

\smallskip
The study on $c$-nilpotent multipliers of Lie algebras have seen several fruitful developments \cite{araskhan,sal2018,ris2014,Joh2014,ar2013}. The main work on $c$-nilpotent multiplier of Lie algebras is to find an upper bound for the dimension of $c$-nilpotent multiplier and then classifying finite dimensional nilpotent Lie algebras under certain conditions. In \cite{nr2016}, an upper bound for $\mathcal{M}^{(2)}(L)$ of finite dimensional Lie algebra $L$ has been investigated using the result of \cite{peyman} and some characterization of $2$-nilpotent multiplier of Heisenberg Lie algebras is given.

The aim of this paper is to generalize the notion of $c$-nilpotent multiplier of Lie algebras to the case of finite dimensional Lie superalgebras. Moreover, we give some upper bounds for $\mathcal{M}^{(2)}(L)$. In particular, we have found the dimension of $2$-nilpotent multiplier of finite dimensional nilpotent Lie superalgebras whose derived subalgebras have dimension at most one. Then find an upper bound for $\mathcal{M}^{(2)}(L)$, i.e., for any nilpotent Lie superalgebra $L=L_{\bar{0}}\oplus L_{\bar{1}}$ of dimension $(k\mid l)$ with $\dim L^{2}=(r \mid s),~r+s\geq 1$, then $$\dim \mathcal{M}^2(L) \leq \frac{1}{3}(k+l-r-s)[(k+l+2r+2s-2)(k+l-r-s-1)+3(r+s-1)]+3. $$ In the last section, we have defined $2$-capability of Lie superalgebra. Then discuss the $2$-capability for the Heisenberg and abelian Lie superalgebras. Precisely, we have proved that $H(m,n)$ is $2$-capable if and only if $m=1,n=0$. Also, the odd Heisenberg Lie superalgebras $H_{m}$ is $2$-capable only for $m=1$. Moreover, $A(m\mid n)$ is $2$-capable if and only if $m+n \geq 2$.

\section{Auxiliary Results}
Throughout this article, for superdimension of a Lie superalgebra $L=L_{\overline{0}} \oplus L_{\overline{1}}$ we simply write $\dim(L)=(m \mid n)$, where $\dim L_{\overline{0}} = m$ and $\dim L_{\overline{1}} = n$. Non-zero elements of $L_{\overline{0}} \cup L_{\overline{1}}$ are called homogeneous elements. For a homogeneous element $v \in L_{\sigma}$, with $\sigma \in \mathbb{Z}_{2}$ we set $|v| = \sigma$ is the degree of $v$. We denote an abelian Lie superalgebra of dimension $(m\mid n)$ by $A(m \mid n)$. A finite dimensional Lie superalgebra $L$ is said to be  Heisenberg Lie superalgebra if $Z(L)=L'$ and $\dim Z(L)=1$. A Heisenberg Lie superalgebra is a nilpotent Lie superalgebra of nilindex $2$. Since the dimension of center of $L$ is one, by assuming the homogeneous generator of $Z(L)$ over an algebraically closed field, Heisenberg Lie superalgebra can be divided into two types, i.e., one with even center and another with odd center. Heisenberg Lie superalgebra with even center is known as special Heisenberg Lie superalgebra and we denote $H(m,n)$ as special Heisenberg Lie superalgebra of dimension $(2m+1\mid n)$. An odd Heisenberg Lie superalgebra is denoted by $H_{m}$ with dimension $(m\mid m+1)$ \cite{Nayak2019,MC2011}.

Finally we adopt the convention that whenever the degree function appeared in a formula, the corresponding elements are supposed to be homogeneous. Now we list some useful results which we use in the next section.

\begin{lemma}\label{l1} (See \cite{RSK2019}, Lemma 6.1)
For any finite dimensional Lie superalgebra $L$, we have $Z^{*}(L)=Z^{\wedge}(L).$
\end{lemma}
\textbf{Remark:}
$L$ is capable if and only if $Z ^{\wedge}(L)=0$. 

\begin{lemma} \label{th4}(See\cite{RSK2019}, Theorem 6.3)
$A(m \mid n)$ is capable if and only if $m=0, n=1$ and $m+n \geq 2$.
\end{lemma}

\begin{lemma} \label{th5}(See \cite{RSK2019}, Theorem 6.4)
$H(m \mid n)$ is capable if and only if $m=1, n=0$.
\end{lemma}

\begin{lemma}\label{th6a}(See \cite{RSK2019}, Theorem 6.7)
$H_{m}$ is capable if and only if $m=1$.
\end{lemma}
\begin{lemma}\label{th5a}(See \cite{Nayak2019}, Proposition 3.4 )
Let $L$ be a nilpotent Lie superalgebra of dimension $(k \mid l)$ with $\dim L'=(r \mid s)$, where $r+s=1$. If $r=1, s=0$ then $L \cong H(m,n)\oplus A(k-2m-1 \mid l-n)$ for $m+n\geq 1$. If $r=0, s=1$ then $L \cong H_{m} \oplus A(k-m \mid l-m-1)$. 
\end{lemma}

\smallskip

Let $X=X_{\bar{0}} \cup X_{\bar{1}}$ be a totally ordered $\mathbb{Z}_{2}$-graded set. Let $\Gamma(X)$ be the groupoid of non-associative monomials in the alphabet $X$, $u \circ v=(u)(v)$ for $u,v \in \Gamma (X)$, and $S(X)$ be the free semigroup of associative words with the bracket removing homomorphism $-: \Gamma (X)\longrightarrow S(X) $. For $u=x_{1}\ldots x_{n} \in S(X),~x_{i} \in X$, we consider the word length $l_{X}(u)=n$, the multi degree $m(u)$, $|u|=\sum^{n} _{i=1} |x_{i}| \in \mathbb{Z}_{2}$. Now let $K$ be a commutative ring with $1$, and let $A(X)$ and $F(X)$ be the free associative and non-associative $K$-algebras respectively. Let $A(X)_{\sigma}, F(X)_{\sigma}$  for $\sigma \in \mathbb{Z}_{2}$ be the $K$-linear spans of the subsets $S(X)_{\sigma}$ and $\Gamma(X)_{\sigma}$ respectively, $A(X)$ and $F(X)$ being the $\mathbb{Z}_{2}$-graded associative and non-associative algebras respectively. Let $I$ be the ideal generated by the homogeneous elements of the form $x\circ y - (-1)^{|x||y|} y\circ x$ and $(x\circ y)\circ z-x\circ (y \circ z)-(-1)^{|x||y|} y \circ (x\circ z)$, for $x, y \in \Gamma(X)$ then $L(X)=F(X)/I$, is the free Lie $K$-superalgebra (see \cite{YAM2000}).

\smallskip

Suppose the set $S(X)$ is ordered lexicographically.
A monomial $u \in \Gamma (X)$ is said to be regular if either $u \in X$ or;
\begin{enumerate}
\item $u=u_{1}\circ u_{2}$ where $u_{1},~ u_{2}$ are regular monomials with $\overline{u_{1}}> \overline{u_{2}}$,
\item $u=(u_{1}\circ u_{2})\circ u_{3}$ with $\overline{u_{2}} \leq \overline{u_{3}}$. 
\end{enumerate}

A monomial $u \in \Gamma (X)$ is said to be $s$-regular if either $u$ is a regular monomial or $u=(v)(v)$ with $v$ a regular monomial and $|v|=1$. Then the set of all images of the s-regular monomials form a basis of the free Lie superalgebra $L(X)$. The analogue of Witt's formula can be seen in Corollary 2.8 in \cite{YAM2000}. 
\begin{theorem}\label{t11}
Let $X=X_{\bar{0}} \cup X_{\bar{1}}$, $X_{\bar{0}}=\{x_{1},\ldots,x_{m} \},~X_{\bar{1}}=\{y_{1},\ldots,y_{n} \}$ be a totally ordered $\mathbb{Z}_{2}$-graded set and $L(X)$ be the free Lie superalgebra, $\mu (l)$ the M\"{o}bius function, and $W(\alpha_{1},\ldots,\alpha_{m+n})$ the rank of the free module of elements of multi degree $\alpha=(\alpha_{1},\ldots,\alpha_{m+n})$ in the free Lie algebra of rank $m+n$, 
$$ W(\alpha_{1},\ldots,\alpha_{m+n})=~\dfrac{1}{|\alpha|}\sum _{e|\alpha}\mu(e)\dfrac{(|\alpha|/e)!}{(\alpha/e)!},$$
where $|\alpha|=\sum ^{m+n}_{i=1}\alpha_{i}$. Let $SW(\alpha_{1},\ldots,\alpha_{m+n})$ be the rank of the free module of elements of multi degree $\alpha=(\alpha_{1},\ldots,\alpha_{m+n})$ in the free Lie superalgebra $L(X)$ of rank $m+n$. Then 
$$ SW(\alpha_{1},\ldots,\alpha_{m+n})=W(\alpha_{1},\ldots,\alpha_{m+n})+~\beta W\left(\frac{\alpha_{1}}{2},\ldots,\frac{\alpha_{m+n}}{2}\right), $$
where
$$
\beta=
\begin{cases}
 0\;\;if~ there~exists~an~i~such ~that~ \alpha_{i}~ is~odd,~or~if~\frac{1}{2}\sum ^{m+n}_{i=m+1}\alpha_{i}|\alpha_{i}|=1, \\
1\;\; otherwise.
   \end{cases}
$$

\end{theorem}

\begin{corollary}\label{c12}\cite{pet2000}
	\begin{align*}
		\begin{split}
	& dim L_r=\frac{1}{r}\sum_{a|r}^{}\mu(a)(m-(-1)^an)^{r/a}\\
	& dim L_{r,+}=\frac{1}{r}\sum_{a|r}^{}\mu(a)\frac{(m-(-1)^an)^{r/a}+(m-n)^{r/a}}{2}\\
	& dim L_{r,-}=\frac{1}{r}\sum_{a|r}^{}\mu(a)\frac{(m-(-1)^an)^{r/a}-(m-n)^{r/a}}{2}\\
	& sdimL_r=\frac{1}{r}\sum_{a|r}^{}\mu(a)(m-n)^{r/a}\\
	& dimL_\alpha=\frac{(-1)^{|\alpha_-|}}{|\alpha|}\sum_{a|r}^{}\mu(a)\frac{(|\alpha|/a)!(-1)^{|\alpha_-|/a}}{(\alpha_1/a)!\cdots(\alpha_{m+n}/a)!}
	\end{split}
	\end{align*}

\end{corollary}


\begin{theorem}\label{l11}
Let $F$ be a free Lie superalgebra on a $\mathbb{Z}_{2}$-graded set $X$, then $F^{n}/F^{n+i}$ is an abelian Lie superalgebra with the basis of all s-regular monomials on $X$ of lengths $n,n+1,\ldots,n+i-1$ for all $1\leq i\leq n$. In particular, $F^{n}/F^{n+1}$ is an abelian Lie superalgebra of dimension $\sum_{|\alpha|=n}SW(\alpha)$, where $F^{n}$ is the n-th term of the lower central series of $F$.

\end{theorem}

\subsection{Non-abelian tensor product}
Recently, Garc\'{i}a-Mart\'{i}nez \cite{GKL2015} introduced the notion of non-abelian tensor product of Lie superalgebras and exterior product of Lie superalgebras. Here we recall some of the known notations and results from \cite{GKL2015}.

 \smallskip

Let $P$ and $M$ be two Lie superalgebras, then by an action of $P$ on $M$ we mean a $\mathbb{K}$-bilinear map of even grade, \[P\times M \longrightarrow M,~~~~~~~(p, m)\mapsto {}^ pm, \]
such that 
\begin{enumerate}
\item ${}^{[p , p']} m = {}^p({}^{p'}m)-(-1)^{|p||p'|}~  {}^{p'}({}^{p}m),$

\item ${}^p{[m , m']}=[{}^pm, m']+(-1)^{|p||m|}[m , {}^pm'],$

\end{enumerate}
for all homogeneous $p, p' \in P$ and $m, m' \in M$. For any Lie superalgebra $M$, the Lie multiplication induces an action on itself via ${}^mm'=[m , m']$.  The action of $P$ on $M$ is called trivial if ${}^pm=0$ for all $p \in P$ and $m \in M$. \\
 Given two Lie superalgebras $M$ and $P$ with action of $P$ on $M$, we define the semidirect product $M \rtimes P$ with underlying supermodule $M \oplus P$ endowed with the bracket given by
$[(m, p), (m', p')]=([m, m']+ {}^pm'-(-1)^{|m||p'|}({}^{p'}m), [p, p'])$.\\
A crossed module of Lie superalgebras is a homomorphism of Lie superalgebras $\partial : M\longrightarrow P$ with an action of $P$ on $M$ satisfying 
\begin{enumerate}
\item $\partial ({}^p{m})=[p,\partial(m)],$

\item ${}^{\partial(m)}{m'}=[m, m'],$
\end{enumerate} 
for all $p \in P$ and $m, m' \in M$.

Let $M$ and $N$ be two Lie superalgebras with actions on each other. Let $X_{M , N}$ be the $\mathbb{Z}_{2}$-graded set of all symbols $m\otimes n$, where $m \in M_{\overline{0}}\cup  M_{\overline{1}}$, $n \in N_{\overline{0}}\cup  N_{\overline{1}}$ and the $\mathbb{Z}_{2}$-gradation is given by $|m\otimes n|=|m|+|n|$. The non-abelian tensor product of $M$ and $N$, denoted by $M \otimes N$, as the Lie superalgebra generated by $X_{M , N}$ and subject to the relations:
\begin{enumerate} 
\item $\lambda (m \otimes n)=\lambda m \otimes n= m \otimes \lambda n$,
\item $(m + m')\otimes n= m\otimes n +m'\otimes n$,  where $m , m'$ have the same grade,\\
      $m \otimes (n + n')=m \otimes n + m \otimes n'$,  where $n , n'$ have the same grade,
      
\item $[m , m']\otimes n= (m\otimes {}^ {m'}{n}) -(-1)^{|m||m'|}(m'\otimes {}^{m}{n})$,\\
      $m\otimes [n,n']=(-1)^{|n'|(|m|+|n|)}({}^{n'}{m}\otimes n)-(-1)^{|m||n|}({}^{n}{m}\otimes n')$,

\item $[m\otimes n,m'\otimes n']=-(-1)^{|m||n|}({}^{n}{m} \otimes {}^{m'}{n'}),$      
\end{enumerate}
for every $\lambda \in \mathbb{K}, m, m' \in M_{\overline{0}}\cup  M_{\overline{1}}$ and $n , n' \in N_{\overline{0}}\cup  N_{\overline{1}}$. If $M=M_{\overline{0}}$ and $N=N_{\overline{0}}$ then $M \otimes N$ is the non-abelian tensor product of Lie algebras introduced and studied \cite{ellis}.

 Actions of Lie superalgebras $M$ and $N$ on each other are said to be compatible if 
\begin{enumerate}
\item ${}^{({}^{n}{m})}{n'}=-(-1)^{|m||n|}[{}^{m}{n},n']$,
\item ${}^{({}^{m}{n})}{m'}=-(-1)^{|m||n|}[{}^{n}{m},m']$,
\end{enumerate}
for all $m, m' \in M_{\overline{0}}\cup  M_{\overline{1}}$ and $n, n' \in N_{\overline{0}}\cup  N_{\overline{1}}$.
For instance if $M$, $N$ are two graded ideals of a Lie superalgebra then the actions induced by the bracket are compatible.

\smallskip

 Suppose $M$ and $N$ are Lie superalgebras acting compatibly with each other. There are two Lie superalgebra homomorphisms $\mu:M \otimes N \longrightarrow M$ and $ \nu: M \otimes N \longrightarrow N$ \cite{GKL2015}. As a result of this $[M, N]^{N}$  or $[N, M]^{N}$ is the submodule generated by ${}^ mn$ and by the compatibility condition it is a graded ideal of $N$. Further with some given actions of both $M$ and $N$ on $M \otimes N$, the homomorphisms $\mu, \nu$ are crossed modules.
 
\smallskip

Now on wards, we consider all actions are compatible and we have the following well known results.

\begin{proposition}\label{prop1}\cite[Proposition 3.5]{GKL2015}
Let $M$ and $N$ be Lie superalgebras acting on each other. Then the canonical map $M\otimes _{\mathbb{K}}N \rightarrow M\otimes N,\; m\otimes n \mapsto m \otimes n$, is an even, surjective homomorphism of supermodules.
\end{proposition}

Further the result below tells us when the surjective homomorphism in Proposition \ref{prop1} is an isomorphism.

\begin{proposition}\label{prop2}\cite[Proposition 3.5]{GKL2015}
 If the Lie superalgebras $M$ and $N$ act trivially on each other, then $M\otimes N$ is an abelian Lie superalgebra and there is an isomorphism of supermodules 
 \[M\otimes N \cong M^{ab}\otimes _{\mathbb{K}}N^{ab},\]
where $M^{ab}=M/[M, M]$ and $N^{ab}=N/[N, N]$. 
\end{proposition}

Suppose $L$ and $M$ are the abelian Lie superalgebras acting on each other trivially. Then we denote $L\otimes_{\mathbb{Z}}M$ is the usual tensor product. The following is an immediate consequence of the above lemma.
\begin{corollary}\label{80}
$A(m \mid n)\otimes A(r \mid s) \cong A(m \mid n)\otimes _{\mathbb{Z}}A(r \mid s)\cong A(mr+ns \mid ms+nr)$
\end{corollary}

\begin{lemma}\label{lem4b}\cite[Lemma 6.1]{GKL2015}
Let $M \square N$ be the submodule of $M \otimes N$ generated by elements
\begin{enumerate}
\item $m\otimes n + (-1)^{|m'||n'|}m'\otimes n'$, where $\partial (m)=\partial '(n')$ and $\partial (m')=\partial '(n)$
\item $m_{0} \otimes n_{0}$, where $\partial (m_{0})=\partial '(n_{0})$,
\end{enumerate} 
with $m,m' \in M_{\overline{0}} \cup M_{\overline{1}},~n,n' \in N_{\overline{0}} \cup N_{\overline{1}},~m_{0} \in M_{\overline{0}}$ and $n_{0} \in N_{\overline{0}}$. Then $M \square N$ is a graded central ideal of $M \otimes N$.
\end{lemma}

\begin{definition}
Let $P$ be a Lie superalgebra and $(M , \partial)$ and $(N , \partial^{'})$ two crossed $P$-modules. Then the exterior product of $M$ and $N$ is denoted as $M \wedge N$ and is defined as \[M\wedge N= \frac{M\otimes N}{ M\square N}.\] 
\end{definition}

Let $(M, \partial)$ and $(N, \partial^{'})$ two crossed $P$-modules. Let $L$ be a Lie superalgebra. An even bilinear map $\rho:M\times N \rightarrow L$ is called Lie super exterior pairing if the following holds:
\begin{enumerate}
\item $\rho([m , m'], n)= \rho(m, {}^ {m'}{n}) -(-1)^{|m||m'|}\rho(m', {}^{m}{n})$;
\item $\rho(m, [n,n'])=(-1)^{|n'|(|m|+|n|)}\rho({}^{n'}{m}, n)-(-1)^{|m||n|}\rho({}^{n}{m}, n')$;
\item $[\rho(m, n),\rho(m', n')]=-(-1)^{|m||n|}\rho({}^{n}{m}, {}^{m'}{n'})$; 
\item $\rho(m, n) + (-1)^{|m'||n'|}\rho(m', n')=0$ if $\partial (m)=\partial '(n')$ and $\partial (m')=\partial '(n)$;
\item $\rho(m_{0}, n_{0})=0$ if $\partial (m_{0})=\partial '(n_{0})$;
\end{enumerate} 
for every  $m, m' \in M_{\overline{0}}\cup  M_{\overline{1}}$; $n, n' \in N_{\overline{0}}\cup  N_{\overline{1}};~m_{0} \in M_{\overline{0}}$ and $n_{0} \in N_{\overline{0}}$. A Lie super exterior pairing $\rho:M\times N \rightarrow L$ is called universal whenever there is any other Lie super pairing $\rho':M\times N \rightarrow Q$, then there exist a unique Lie superalgebra homomorphism $\tau: L \rightarrow Q$ such that $\tau \rho = \rho'$.


Let us consider the category $SLie^{2}_{\mathbb{F}}$ \cite{GKL2015}. Let the Lie superalgebras $M$ and  $N$, $P$ and $Q$ act compatibly on each other. Also $\phi:M\longrightarrow P$ and $\psi:N\longrightarrow Q$ be two Lie superalgebra homomorphisms which preserve the action, i.e.,
\[\phi({}^ {n}{m})={}^ {\psi(n)}{\phi(m)},~~~~~\psi({}^ {m}{n})={}^ {\phi(m)}{\psi(n)}.\]
Then we have a homomorphism $\phi \otimes \psi: M\wedge N \longrightarrow P \wedge Q$ defined by $m\wedge n \mapsto \phi (m)\wedge \psi(n).$

\begin{proposition}\label{prop4}\cite[Proposition 3.8]{GKL2015}
Given an exact sequence of Lie superalgebras
$$(0, 0) \longrightarrow (K,L) \overset{(i, j)} \longrightarrow (M, N) \overset{(\phi, \psi)} \longrightarrow (P, Q) \longrightarrow (0, 0)$$ there is an exact sequence of Lie superalgebras
$$
(K \wedge M) \rtimes (M \wedge L) \overset{\alpha} \longrightarrow M \wedge N \overset{\phi \otimes \psi} \longrightarrow P \wedge Q \longrightarrow 0.$$
\end{proposition}
Any ideal $M$ of the Lie superalgebra $L$ act on $L$ via Lie multiplication. Thus, we have the following straightforward result. 
\begin{lemma}\label{121}
Let $M$ be an ideal of the Lie superalgebra $L$ and the epimorphism $\theta: M \wedge L \longrightarrow [M,L]$ is define by $m\wedge l \mapsto [m,l]$. Then the Lie superalgebras $L$ and $M \wedge L$ act compatibly on each other as follows
\[{}^ {l'}{m\wedge l}={}^ {l'}{m}\wedge l +(-1)^{|l'||m|}m\wedge{}^ {l'}{l},~~~~~~~{}^ {x}{l}={}^ {\theta(x)}{l},\]
where $x \in M\wedge L$, $l'$ and $m$ are homogeneous elements of $L$ and $M$ respectively.
\end{lemma}

Now, following Lemma $\ref{121}$, we can define the exterior product $(M\wedge L)\wedge L$. Thus inductively we can construct the exterior product $$M \wedge ^{c+1}L=(\ldots((M\wedge L)\wedge L)\cdots \wedge L),~~ c\geq 1$$ and for $j\geq 1$  the epimorphism $\theta_{j}: M \wedge^{j} L \longrightarrow [M,_{j}L]$ given by $$(m\wedge l_{1}\wedge \cdots \wedge l_{j})\mapsto [m,l_{1},\ldots,l_{j}],$$
where $[m,l_{1},\ldots,l_{j}]=[\ldots[[m,l_{1}],l_{2}]\ldots,l_{j}]$ and $\theta_{1}=\theta$. Similarly, if $K$ is any other ideal of $L$ with $[M,K]=0$, then one cane define the non-abelian exterior product $M \wedge ^{c+1}(L/K)$

\section{Bounds for $\mathcal{M}^{(c)}(L)$}
Now we are ready to define $c$-nilpotent multiplier of a Lie superalgebra and obtain some of its bounds analogous to the case of Lie algebra. 

Let $L=L_{\bar{0}}\oplus L_{\bar{1}}$ be a Lie superalgebra and $F$ be the free Lie superalgebra such that $0\rightarrow R \rightarrow F \rightarrow L \rightarrow 0$ be a free presentation of $L$. For $c\geq 1$, we define the $c$-multiplier of $L$ to be  $$ \mathcal{M}^{(c)}(L) = R\cap \gamma_{c+1}(F)/ \gamma_{c+1}(R,F)   $$ 
where $\gamma_{c+1}(F)$ is the $(c+1)$-th term of the lower central series of $F$ and $\gamma_{1}(R,F)=R$, $\gamma_{c+1}(R,F)=[\gamma_{c}(R,F),F]$. In particular, if $c=1$, then $\mathcal{M}^{(1)}(L)=\mathcal{M}(L)$ is Schur multiplier of Lie superalgebra $L$ which has been studied in \cite{Nayak2019}. Following the definition, $\mathcal{M}^{(c)}(L)$ is an abelian Lie superalgebra.

 Let $L=L_{\overline{0}}\oplus L_{\overline{1}}$ be a Lie superalgebra. Let us define the set $Z_n(L)=\{x\in L\mid [x,y]\in Z_{n-1}(L)~\forall ~y\in L\}$ where $Z_0(L)={0}$. Then for each $n$, $Z_n(L)$ is a graded ideal of $L$. In particular, $Z_{n-1}(L)$ is a graded ideal of $Z_{n}(L)$. If $n=1$, then we will get the center of $L$. Furthermore, the upper central series of $L$ is defined as $$\{0\}=Z_0(L)\subseteq Z_1(L)\subseteq \dots  \subseteq Z_n(L)$$ where 	$Z_1(L)=Z(L)$. One can verify that $Z\bigg(\frac{L}{Z_n(L)}\bigg)=\frac{Z_{n+1}(L)}{Z_n(L)}$. 
 
 Let $K$ be a graded ideal of a Lie superalgebra $L$. $K$ is said to be $n$-cenrtal if $K \subseteq Z_n(L)$. Evidently, $\gamma _{n+1}(K,L)={0}$, then $K \subseteq Z_n(L)$. Thus, if $L$ is a nilpotent Lie superalgebra of nilindex $n$, then $Z_n(L)=L$.

\begin{proposition}\label{222}
Let $L=L_{\bar{0}}\oplus L_{\bar{1}}$ be a Lie superalgebra with a free presentation $0\rightarrow R\rightarrow F\xrightarrow{\pi} L \rightarrow 0$, and $M$ be a graded ideal of $L$. Define $M \cong N/R$ for some graded ideal $N$ of $F$, and $T_i=\gamma_i(N,F)/\gamma_i(R,F)~ (i\geq 1)$. Then
\begin{enumerate}
\item $T_{i+1}$ is a homomorphic image of $M\wedge^i L$.
\item If $M$ is $n$-central and $P= L/\gamma_{n+1}(L)$, then $T_{i+1}$ is a homomorphic image of $M\wedge ^i P$.
\end{enumerate}	
\end{proposition}
\begin{proof}
(1) Let $l$ and $x$ be the homogeneous elements of $L$ and $\gamma _i(N,F)$ respectively. Let $\pi^{-1}(l)=y$, then we have the following well-defined action of $L$ and $T_j$, acting on each other compatibly,
\begin{equation}\label{eq1}
^l(x+\gamma_i(R,F))=[y,x]+\gamma_i(R,F) ~{\rm{and}} ~ ^{(x+\gamma_i(R,F))}l=\pi([x,y])
\end{equation}
Since $\gamma _i(R,F)\subseteq \ker\pi$, thus $\varphi:T_i\rightarrow L$	is the induced map by $\pi$. Moreover, $\varphi$ is a crossed module with the action $\ref{eq1}$. If   $\xi _i: T_i\times L \rightarrow T_{i+1}$ is defined by $\xi_i(x+\gamma _i(R,F),l)=[x,y]+\gamma _{i+1}(R,F)$, then $\xi_i$ is a Lie super exterior pairing. Thus from universal property, there exists a unique Lie superalgebra homomorphism $\overline{\xi}_i:T_i \wedge L\rightarrow T_{i+1} $. By Proposition $\ref{prop4}$, $\hat{\xi}_i\wedge Id_L:(T_i\wedge L)\wedge L\rightarrow T_{i+1}\wedge L$ is an onto homomorphism, where $Id_L:L\rightarrow L$. Now the proof follows using induction hypothesis.
	

(2) For any graded ideal $I$ of $F$, using induction it can be shown that $[I,\gamma _{i+1}(N,F)]\subseteq \gamma _{i+1}(I,F)$. Now observe that $[\gamma _i(N,F),\gamma _{n+1}(F)]\subseteq \gamma _{i+n+1}(N,F) \subseteq \gamma _{i+1}(R,F)$.	
Thus from \ref{eq1}, $\gamma _{n+1}(F)$ acts trivially on $T_i$ and the induce action of ${L}/ {\gamma _{n+1}(L)}$ and $T_i$ on each other is compatibly. Thus from part (1) the result follows.	
\end{proof}

 \begin{corollary}\label{333}
By the assumption and notations in Proposition $\ref{222}$, the factor Lie superalgebra $R\cap \gamma_{i+1}(N,F)/\gamma_{i+1}(R,F)$ is a homomorphic image of ${\rm{ker}}(\theta_{i})$.
 \end{corollary}

\begin{proof}
From proposition $\ref{222}$, there exists surjective homomorphism $\psi_i: M\wedge^iL \rightarrow T_{i+1}$. Then the result follows from the following commutative diagram
\begin{center}
\begin{tikzpicture}[>=latex]
	\node (A_{1}) at (0,0) {\(0\)};
	\node (A_{2}) at (2,0) {\(ker\theta_i\)};
	\node (A_{3}) at (4,0) {\(M\wedge^iL\)};
	\node (A_{4}) at (6,0) {\([M,_iL]\)};
	\node (A_{5}) at (8,0) {\(0\)};
	\node (B_{1}) at (0,-2) {\(0\)};
	\node (B_{2}) at (2,-2) {\(\frac{R\cap \gamma_{i+1}(N,F)}{\gamma_{i+1}(R,F)}\)};
	\node (B_{3}) at (4,-2) {\(T_{j+1}\)};
	\node (B_{4}) at (6,-2) {\(\frac{\gamma_{i+1}(N,F)}{R\cap \gamma_{i+1}(N,F)}\)};
	\node (B_{5}) at (8,-2) {\(0\)};
	\draw[->] (A_{1}) -- (A_{2});
	\draw[->] (A_{2}) -- (A_{3});
	\draw[->] (A_{3}) -- (A_{4});
	\draw[->] (A_{4}) -- (A_{5});
	\draw[->] (B_{1}) -- (B_{2});
	\draw[->] (B_{2}) -- (B_{3});
	\draw[->] (B_{3}) -- (B_{4});
	\draw[->] (B_{4}) -- (B_{5});
	\draw[->] (A_{3}) -- (A_{4}) node[midway,above] {$\theta_i$};
	\draw[->] (A_{2}) -- (B_{2}) node[midway,right] {$\psi_i|$};
	\draw[->] (A_{3}) -- (B_{3}) node[midway,right] {$\psi_i$};
	\draw[->] (A_{4}) -- (B_{4}) node[midway,right] {$\cong$};
	
\end{tikzpicture}
\end{center}

where $\psi_i|$ is the restriction of $\psi_i$ on ${\rm{ker}}(\theta_{i})$. 
\end{proof}

From $\ref{333}$, we have the exact sequences of $c$-nilpotent multipliers of Lie superalgebras.
\begin{proposition}\label{p11}
Let $M$ be a graded ideal of a Lie superalgebra $L=L_{\bar{0}}\oplus L_{\bar{1}}$. Then the following sequences are exact:
\begin{enumerate}
\item ${\rm{ker}}(\theta_c)\rightarrow \mathcal{M} ^{(c)}(L)\rightarrow \mathcal{M}^{(c)}(\frac{L}{M})\rightarrow \frac{M\cap \gamma_{c+1}(L)}{\gamma_{c+1}(M,L)}\rightarrow 0$.
\item If $M$ is $n$-central ideal with $n\leq c$, $M\wedge^c \frac{L}{\gamma_{n+1}(L)}\rightarrow \mathcal{M}^{(c)}(L)\xrightarrow{\varphi} \mathcal{M}^{(c)}(\frac{L}{M})\rightarrow M\cap \gamma_{c+1}(L)\rightarrow 0$.
\end{enumerate}
\end{proposition}
\begin{proof}
Consider a free presentation of $L$, i.e., $0\rightarrow R\rightarrow F\rightarrow L\rightarrow 0 $. Then, there exists a graded ideal $N$ of $F$ such that $M\cong N/R$. Evidently, the following inclusion maps 
	$$R\cap \gamma_{c+1}(N,F)\xrightarrow{\subseteq}R\cap \gamma_{c+1}(F)\xrightarrow{\subseteq} N\cap \gamma_{c+1}(F)\xrightarrow{\subseteq} (N\cap \gamma_{c+1}(F))+R,$$  induce the exact sequence of homomorphisms, i.e.,
	$$0\rightarrow \frac{R\cap \gamma_{c+1}(N,F)}{\gamma_{c+1}(R,F)}\rightarrow \frac{R\cap \gamma_{c+1}(F)}{\gamma_{c+1}(R,F)} \rightarrow \frac{N\cap \gamma_{c+1}(F)}{\gamma_{c+1}(N,F)}\rightarrow \frac{(N\cap \gamma_{c+1}(F))+R}{\gamma_{c+1}(N,F)+R}\rightarrow 0.$$
If $M$ is $n$-central with $n\leq c$, then $\gamma_{c+1}(N,F)\subseteq \gamma_{n+1}(N,F)\subseteq R$. Now, the result follows from Proposition $\ref{222}$ and Corollary $\ref{333}$.
\end{proof}

\begin{corollary}\label{1200}
If $M$ is $n$-central ideal of a Lie superalgebra $L=L_{\bar{0}}\oplus L_{\bar{1}}$ with $n\leq c$, then $$M\wedge^{c}L \rightarrow \mathcal{M} ^{(c)}(L)\rightarrow \mathcal{M}^{(c)}(\frac{L}{M})\rightarrow \frac{M\cap \gamma_{c+1}(L)}{\gamma_{c+1}(M,L)}\rightarrow 0.$$
\end{corollary}
\begin{proof}
Since $M$ is $n$-central ideal, ${\rm{ker}}(\theta_c)=M\wedge^{c}L$.
\end{proof}

Using the Proposition $\ref{p11}$, the following inequalities can be obtained.
\begin{corollary}
Let $M$ be a graded ideal of a finite dimensional Lie superalgebra $L=L_{\bar{0}}\oplus L_{\bar{1}}$. Then
\begin{enumerate}
\item The Lie superalgebra $\mathcal{M}^{(c)}(L)$ is finite dimensional.
\item $\dim(\mathcal{M}^{(c)}(\frac{L}{M}))\leq \dim(\mathcal{M}^{(c)}(L))+\dim(\frac{M\cap \gamma_{c+1}(L)}{\gamma_{c+1}(M,L)}).$
\item Following the notations of Proposition $\ref{222}$, $\dim(\mathcal{M}^{(c)}(L)) + \dim(M\cap \gamma_{c+1}(L))= \dim(\mathcal{M}^{(c)}(\frac{L}{M}))+\dim(\frac{\gamma_{c+1}(N,F)}{\gamma_{c+1}(R,F)})$ , 
\item If $\mathcal{M}^{(c)}(L)=\{0\}$, then $\mathcal{M}^{(c)}(\frac{L}{M})\cong \frac{M\cap \gamma_{c+1}(L)}{\gamma_{c+1}(M,L)}$.
\item If $M$ is $n$-central with $n\leq c$, then $$\dim(\mathcal{M}^{(c)}(L))+\dim(M\cap \gamma_{c+1}(L))\leq \dim(\mathcal{M}^{(c)}(\frac{L}{M}))+\dim\bigg(M\wedge^c \frac{L}{\gamma_{n+1}(L)}\bigg).$$
	
\end{enumerate}	
	
\end{corollary}

\begin{lemma}\label{444}
Let $L=L_{\bar{0}}\oplus L_{\bar{1}}$ be a Lie superalgebra with a free presentation $0\rightarrow R\rightarrow F\xrightarrow{\pi} L \rightarrow 0$, and $M$ be a graded ideal of $L$. Define $M \cong N/R$ for some graded ideal $N$ of $F$ and $T=L/M$. Then there exists a Lie superalgebra $P$ with a graded ideal $S$ such that the following hold:
\begin{enumerate}
\item $\gamma _{c+1}(L)\cap M \cong P/S;$ 
\item $S \cong \mathcal{M}^{(c)}(L);$
\item $\mathcal{M}^{(c)}(T)$ is a homomorphic image of $P$.
\item $\gamma _{c+1}\cap M$ is a homomorphic image of $\mathcal{M}^{(c)}(T)$, whenever $M$ is a $c$-central ideal of $L$.
\end{enumerate}		
\end{lemma}

\begin{proof}
	Now
	\begin{align*}
	\gamma _{c+1}(L)\cap M & =((\gamma _{c+1}(F))+R)/R)\cap N/R \\
	& =((\gamma _{c+1}(F)\cap N)+R)/R\\
	&\cong (\gamma _{c+1}(F)\cap N)/(\gamma _{c+1}(F)\cap R)\\
	&\cong \frac{(\gamma _{c+1}(F)\cap N)/\gamma _{c+1}(R,F)} 
	{(\gamma _{c+1}(F)\cap R)/\gamma _{c+1}(R,F)} ~\cong~ P/S
	\end{align*}
Thus $(1)$ and $(2)$ follow by taking $P=(\gamma _{c+1}(F)\cap N)/\gamma _{c+1}(R,F),~ S=(\gamma _{c+1}(F)\cap R)/\gamma _{c+1}(R,F) \cong \mathcal{M}^{(c)}(L)$.
	Furthermore, $$\mathcal{M} ^{(c)}(T)\cong (\gamma _{c+1}(F)\cap N)/\gamma _{c+1}(N,F)\cong\frac {(\gamma _{c+1}(F)\cap N)/\gamma _{c+1}(R,F)}{\gamma _{c+1}(N,F)/\gamma _{c+1}(R,F)}$$
Thus, $\mathcal{M} ^{(c)}(T)$ is a  homomorphic image of $P$. As $M$ is a $c$-central ideal of $L$, $\gamma _{c+1}(N,F)\subseteq \gamma _{c+1}(F)\cap R$. Thus, 
	\begin{align*}
	\gamma _{c+1}(L)\cap B &=((\gamma _{c+1}(F)+R)/R)\cap N/R \\ 
	&\cong (\gamma _{c+1}(F)\cap N)/(\gamma _{c+1}(F)\cap R) \\
	&\cong \frac{(\gamma _{c+1}(F)\cap N)/\gamma _{c+1} (N,F)}{\gamma _{c+1}(F)\cap R/\gamma _{c+1}(N,F)}
	\end{align*}
which completes the proof.
	
\end{proof}

The immediate consequence of the above lemma is as follow.
\begin{corollary}\label{555}
Let $L=L_{\bar{0}}\oplus L_{\bar{1}}$ be a Lie superalgebra with a graded ideal $M$ and $T=L/M$. Then $$\dim~\mathcal{M} ^{(c)}(T)\leq \dim~\mathcal{M} ^{(c)}(L)+\dim~(\gamma _{c+1}(L)\cap M).$$
\end{corollary}

\begin{lemma}\label{666}
Let $L=L_{\bar{0}}\oplus L_{\bar{1}}$ be a Lie superalgebra with a free presentation $0\rightarrow R\rightarrow F\xrightarrow{\pi} L \rightarrow 0$, and $M$ be a graded ideal of $L$ with $M \subseteq Z(L)$. Define $M \cong N/R$ for some graded ideal $N$ of $F$ and $T=L/M \cong F/N$. Then $\gamma_{c+1}(N,F)/(\gamma_{c+1}(R,F)+\gamma_{c+1}(N))$ is a homomorphic image of ${(T/T^2)}^c\otimes M=\underbrace {T/T^2\otimes \ldots \otimes T/T^2}_{\text{c-times}}\otimes M,~c\geq 1$ .
\end{lemma}
\begin{proof}
Define the map $\psi:\underbrace {T/T^2\times \ldots \times T/T^2}_{\text{c-times}}\times M\rightarrow \gamma _{c+1}(N,F)/(\gamma_{c+1}(R,F)+\gamma_{c+1}(N))$ with
 $$\theta (f_1+(F^2+N),\ldots,f_c+(F^2+N),x+R):= [x,f_1,\ldots,f_c]+(\gamma_{c+1}(R,F)+\gamma_{c+1}(N)),$$ 
where $f_1,f_2,\ldots,f_c \in F_{\bar{0}}\cup F_{\bar{1}}$; $x\in N_{\bar{0}}\cup N_{\bar{1}}$. We will show that $\theta$ is well defined. Note that $$T/T^2\cong F/(F^2+N)~~ {\rm{and}}~~ M/M^2\cong N/R .$$ Suppose $(f_1+F^2+N,f_2+F^2+N,\ldots,f_c+F^2+N,x+R)=(f'_1+F^2+N,\ldots,f'_c+F^2+N,y+R),~y\in N_\alpha$ for all homogeneous elements. Then $ f_i-f'_i\in F^2+N, ~{\rm{for}}~ i=1,2,\ldots,c ~{\rm{and}}~ x-y \in R_{\bar{0}}\cup R_{\bar{1}}.$ Which implies $ f_i-f'_i=g_i+s_i $, where $g_i$ and $s_i$ are  homogeneous elements of $ F^2 $ and $N$ respectively.
	
Again, $[x,f_1,f_2,\ldots,f_c]+\gamma_{c+1}(R,F)+\gamma_{c+1}(N)=[y,f_1',f'_2,\ldots,f_c']+\gamma_{c+1}(R,F)+\gamma_{c+1}(N)$ which implies $[x,f_1,f_2,\ldots,f_c]-[y,f_1',f_2',\ldots,f_c']\in \gamma_{c+1}(R,F)+\gamma_{c+1}(N)$. By the assumption $M$ is contained in $Z(L)$, which implies $[N,F]\subseteq R$. Using Jacobi Identity, $[x,f_1,\ldots,f_c]-[x+r,f_1+g_1+s_1,\ldots,f_c+g_c+s_c]\in \gamma_{c+1}(R,F)+\gamma_{c+1}(N)$. Hence $\theta$ is well-defined. Therefore, by universal property there exists a unique homomorphism $$\bar{\psi}:\underbrace{T/T^2\otimes \ldots\otimes T/T^2}_{\text{c-times}}\otimes M\rightarrow \gamma_{c+1}(N,F)/(\gamma_{c+1}(R,F)+\gamma_{c+1}(N))$$ with $Im(\bar{\psi})=\gamma_{c+1}(N,F)/(\gamma_{c+1}(R,F)+\gamma_{c+1}(N))$, as desired.
\end{proof}

\begin{theorem}\label{777}
Let $L=L_{\bar{0}}\oplus L_{\bar{1}}$ be a finite dimensional Lie superalgebra with a free presentation $0\rightarrow R\rightarrow F\xrightarrow{\pi} L \rightarrow 0$, and $M$ be a graded ideal of $L$ with $M \subseteq Z(L)$. Define $T=L/M$, then $$\dim\mathcal{M}^{(c)}(L)+\dim(\gamma_{c+1}(L)\cap M)\leq \dim\mathcal{M}^{(c)}(T)+\dim\mathcal{M}^{(c)}(M)+\dim((T/T^2)^c\otimes M).$$
\end{theorem}
\begin{proof}
There exists a graded ideal $N$ of the free Lie superalgebra $F$ such that $M\cong N/R$, then $[N,F]\subseteq R$. Observe that
	$$\frac{\gamma_{c+1}(N,F)/(\gamma_{c+1}(R,F))}{(\gamma_{c+1}(R,F)+\gamma_{c+1}(N))/(\gamma_{c+1}(R,F))}\cong \frac{\gamma_{c+1}(N,F)}{\gamma_{c+1}(R,F)+\gamma_{c+1}(N)}.$$
Now using Lemma $\ref{444}$, $\dim\mathcal{M}^{(c)}(L)+\dim(\gamma_{c+1}(L)\cap M)=\dim\mathcal{M}^{(c)}(T)+\dim(\gamma_{c+1}(N,F)/(\gamma_{c+1}(R,F))$ and from Lemma $\ref{666}$, we have 
    \begin{align*}
    \dim \mathcal{M}^{(c)}(L)+\dim(\gamma_{c+1}(L)\cap M) &=\dim\mathcal{M}^{(c)}(T)+\dim\bigg(\frac{\gamma_{c+1}(N,F)}{\gamma_{c+1}(R,F)+\gamma_{c+1}(N)}\bigg)\\& +\dim\bigg(\frac{\gamma_{c+1}(R,F)+\gamma_{c+1}(N)}{\gamma_{c+1}(R,F)}\bigg)\\
    &=\dim\mathcal{M}^{(c)}(T)+\dim\mathcal{M}^{(c)}(M)-\dim\bigg(\frac{\gamma_{c+1}(N,F)}{\gamma_{c+1}(R,F)+ \gamma_{c+1}(N)}\bigg)\\& +\dim\bigg(\frac{\gamma_{c+1}(R,F)
    	+ \gamma_{c+1}(N)}{\gamma_{c+1}(R,F)}\bigg)\\
    & \leq \dim\mathcal{M}^{(c)}(T)+\mathcal{M}^{(c)}(M)+\dim((T/T^2)^c\otimes M).
    \end{align*}
\end{proof}

A Lie superalgebra $H$ is said to be a  generalized Heisenberg Lie superalgebra of rank $(r\mid s)$ if $Z(H)=H^2$ with $\dim Z(H)=(r\mid s)$. As an application of the above theorem, we will find an upper bound for the $c$-nilpotent multiplier of generalized Heisenberg Lie superalgebra of rank $(r\mid s)$.

\begin{lemma}\label{l12}
 Let $A(m\mid n)$ be an abelian Lie superalgebra of dimension $(m\mid n)$. Then $\dim(\mathcal{M}^{(c)}(L))=\sum_{|\alpha|=c+1}SW(\alpha)$, where $SW(\alpha_{1},\ldots,\alpha_{m},\ldots,\alpha_{m+1},\cdots,\alpha_{m+n})$ is the rank of the free module of elements of multi degree $\alpha=(\alpha_{1},\ldots,\alpha_{m+n})$ in the free Lie superalgebra $L(X)$ of rank $m+n$.
\end{lemma}
 \begin{proof}
Let $X=X_{\bar{0}} \cup X_{\bar{1}}$, $X_{\bar{0}}=\{x_{1},\ldots,x_{m} \}$, $X_{\bar{1}}=\{y_{1},\ldots,y_{n} \}$ be a $\mathbb{Z}_{2}$-graded set and $F=L(X)$ be the free Lie superalgebra over $X$. Then $0\longrightarrow F^{2}\longrightarrow F \longrightarrow A(m\mid n)\longrightarrow 0$ is a free presentation for $A(m\mid n)$. Therefore, $\mathcal{M}^{(c)}(L) \cong \gamma_{c}(F)/\gamma_{c+1}(F)$. Now the result follows from Theorem $\ref{l11}$.
\end{proof}
\begin{corollary}
	Let $H$ be a generalized Heisenberg Lie superalgebra of dimension $(m|n)$ with rank $(r|s)$. Then $$\dim\mathcal{M}^{(c)}(H)\leq \sum_{|\alpha|=c+1}^{}SW(\alpha)+\sum_{|\alpha'|=c+1}^{}SW(\alpha')+(m+n-r-s)^c(r+s),$$ where $SW(\alpha)$ $(resp.~ SW(\alpha'))$ is the rank of the free module of elements of multi degree $\alpha=(\alpha_1,\dots,\alpha_{m-r},\alpha_{m-r+1},\dots, \alpha_{m-r+n-s})$ $(resp.~ \alpha'=(\alpha',\dots,\alpha'_r, \alpha'_{r+1},\dots,\alpha'_{r+s}))$ in the free Lie superalgebra $L(X_1)$ $(resp.~ L(X_2))$ of rank $m-r+n-s$ $(resp. ~r+s)$.

\end{corollary}

\begin{proof}
	We conclude the result using Theorem $\ref{777}$. To use Theorem $\ref{777}$, take $M=Z(H)$. Then 
	$$\dim(\gamma_{c+1}(H)\cap Z(H))=\begin{cases}
	(r|s)~~~~~~~~{\rm{if}}~~ c=1\\
	0~~~~~~~~~~~~~~~{\rm{if}}~~ c\geq 2
	\end{cases}$$
	
	Now $\dim\bigg(\frac{H}{Z(H)}\bigg)=(m-r|n-s)$. Then using Lemma $\ref{l12}$, 
	\begin{align*}
	\dim \mathcal{M}^{(c)}(H)&\leq \dim \mathcal{M}^{(c)}\bigg(\frac{H}{Z(H)}\bigg)+\dim \mathcal{M}^{(c)}(Z(H))+\dim \bigg(\bigg(\frac{H}{Z(H)}\bigg)^c \otimes Z(H)\bigg)\\
	&=\sum_{|\alpha|=c+1}^{}SW(\alpha)+\sum_{|\alpha'|=c+1}^{}SW(\alpha')+(m+n-r-s)^c(r+s).
	\end{align*}

\end{proof}

\section{2-nilpotent multiplier of Lie superalgebra}
In this section we define direct product of two Lie superalgebras and then study some of its properties relating to $2$-nilpotent multiplier. In particular, we derive the dimension of  $\mathcal{M}^{(2)}(L)$, when $L$ is a nilpotent Lie superalgebra with $\dim L^2 =1$. 

Let $L_{1}$ and $L_{2}$ be two Lie superalgebras with the free presentation
$$ 0\longrightarrow R_{1} \longrightarrow F_{1} \overset{\delta_{1}}{\longrightarrow} L_{1} \longrightarrow 0 $$ and $$ 0\longrightarrow R_{2} \longrightarrow F_{2} \overset{\delta_{2}}{\longrightarrow} L_{2}\longrightarrow 0 $$ 
of $L_{1}$ and $L_{2}$ respectively, where $F_{1}$ and $F_{2}$ are free Lie superalgebras on some $\mathbb{Z}_{2}$-graded set.\\

\begin{definition}
Let $L_{1}$, $L_{2}$ and $L$ be  Lie superalgebras with $\rho_{i}:L_{i}\rightarrow L$ be Lie superalgebra homomorphism for $i=1,2$. Then $(L,(\rho_{1},\rho_{2}))$ is called a free product of $L_{1}$ and $L_{2}$ if, for any Lie superalgebra $K$ with the Lie superalgebra homomorphism $\sigma_{i}:L_{i}\rightarrow K$, $i=1,2$, then there exists a unique Lie superalgebra homomorphism $\pi:L\rightarrow K$ such that $\pi \rho_{i}=\sigma_{i}$, $i=1,2$.
\begin{center}

\begin{tikzpicture}[>=latex]
\node (L_{i}) at (0,0) {\({L_{i}}\)};
\node (K) at (0,-2) {\({K}\)};
\node (L) at (2.5,0) {\({L}\)};
\draw[->] (L_{i}) -- (L) node[midway,above] {$\rho_{i}$};
\draw[->] (L_{i}) -- (K) node[midway,left] {$\sigma_{i}$};
\draw[->,dashed] (L) -- (K) node[midway,below] {$\pi$};
\end{tikzpicture}
\end{center}
The free product of $L_{1}$ and $L_{2}$ is denoted by $L_{1}*L_{2}$.
\end{definition}

\begin{lemma}\label{31}
If  $F_{1}$ and $F_{2}$ are free Lie superalgebras on the $\mathbb{Z}_{2}$-graded sets $X$ and $Y$ respectively, then $F_{1}*F_{2}$ is a free Lie superalgebra on the $\mathbb{Z}_{2}$-graded set $X \cup Y$.
\end{lemma}
The following lemma can be easily concluded from Lemma $\ref{31}$.
\begin{lemma}\label{34}
	If $F=F_{1}*F_{2}$, then
	$$F^3=[F_2,F_1,F_1]+[F_2,F_1,F_2]+F_1^3+F_2^3.$$
\end{lemma}

If we define an ordering on the $\mathbb{Z}_{2}$-graded set $X \cup Y$ by imposing every element of $X$ proceeds each element of $Y$. Then we have the following result.
\begin{theorem}\label{36}
	Let $F_1$ and $F_2$ be two free Lie superalgebras freely generated by $X$ and $Y$, respectively, and $F=F_1*F_2$. Then $[F_2,F_1,F_1]+[F_2,F_1,F_2]({\rm{mod}}~F^4)$ is an abelian Lie superalgebra.
\end{theorem}
\begin{proof}
The proof follows from the Theorem $\ref{l11}$.
\end{proof}

\begin{lemma}\label{32}
	Let $F=F_1*F_2$ be the free product of $F_1$ and $F_2$. Then $$0\longrightarrow R \longrightarrow F \overset{\delta}{\longrightarrow} L_1 \oplus L_2\longrightarrow 0$$ is a free presentation for $L_1\oplus L_2$ in which $R=R_1+R_2+[F_1,F_2]$.
\end{lemma}

\begin{proof}
By definition, we have the Lie superalgebra homomorphism $\delta:F\rightarrow L_{1}\oplus L_{2}$ such that $\delta \rho_{i}=\pi_{i}\delta_{i}$ for $i=1,2$, where $\pi_{i}$ is the projection map from $L_{i}$ into $L_{1}\oplus L_{2}$. We will show that $\ker(\delta)=R_1+R_2+[F_1,F_2]$. To check this, we observe that $\delta_{i}(R_{i})=0$ and for any $x \in F_{1}$, $y \in F_{2}$, then $[\delta(x),\delta(y)]=0$. Thus $R_1+R_2+[F_1,F_2] \subseteq \ker(\delta)$. Conversely, let $x \in \ker(\delta) $, then we can express $x$ as a finite linear combination of basic commutators of $F$. Because $\delta \rho_{i}=\pi_{i}\delta_{i}$; this, in turn, gives $x \in R$, ending the proof.
\end{proof}



\begin{lemma}\label{33}
	Let $L_1$ and $L_2$ be two Lie superalgebras, then $$\mathcal{M}^{(2)}(L_1\oplus L_2)\cong \mathcal{M}^{(2)}(L_1)\oplus\mathcal{M}^{(2)}(L_2)\oplus K $$ for some subalgebra $K$ of $\mathcal{M}^{(2)}(L_1\oplus L_2)$.
\end{lemma}

\begin{proof}
As mention in Lemma $\ref{31}$, $F_{1},F_{2}$ and $F$ are free Lie superalgebras. Thus we have a surjective map 
 $F\rightarrow F_1\times F_2$ and this map induces a surjective homomorphism $$\eta_{1}:\frac{R\cap F^3}{[R,F,F]}\rightarrow \frac{R_1\cap F_1^3}{[R_1,F_1,F_1]}\oplus \frac{R_2\cap F_2^3}{[R_2,F_2,F_2]}.$$
 
Let us define $\eta_{2}: \frac{R_1\cap F_1^3}{[R_1,F_1,F_1]}\oplus \frac{R_2\cap F_2^3}{[R_2,F_2,F_2]}\rightarrow \frac{R\cap F^3}{[R,F,F]}$ by $$\eta_{2}(x+[R_1,F_1,F_1],y+[R_2,F_2,F_2])= x+y+[R,F,F].$$ Since $[R_1,F_1,F_1]+[R_2,F_2,F_2] \subseteq [[R,F],F]$, thus $\eta_{2}$ is a well-defined map and a homomorphism. Also, if $(x+[R_1,F_1,F_1],y+[R_2,F_2,F_2]) \in \frac{R\cap F^3}{[R,F,F]}$ where $x \in R_1\cap F_1^3 $ and $y \in R_2\cap F_2^3$, then 
$$ \eta_{2}\eta_{1}(x+y+[R,F,F])=\eta_{2}(x+[R_1,F_1,F_1],y+[R_2,F_2,F_2])=x+y+[R,F,F].$$ 
That is $\eta_{2}\eta_{1}=Id$. Thus, there exists a subalgebra $K$ of $\mathcal{M}^{(2)}(L_1\oplus L_2)$ such that $\mathcal{M}^{(2)}(L_1\oplus L_2)\cong \mathcal{M}^{(2)}(L_1)\oplus\mathcal{M}^{(2)}(L_2)\oplus K $.

\end{proof}

The aim is to find the complete structure of $K$ in order to derive the behavior of $\mathcal{M}^{(2)}(L_1\oplus L_2)$.  
\begin{lemma}\label{35}
Consider the surjective homomorphism defined in Lemma $\ref{33}$
 $$\eta_{1} :\mathcal{M}^{(2)}(L_1\oplus L_2)\rightarrow \mathcal{M}^{(2)}(L_1)\oplus\mathcal{M}^{(2)}(L_2).$$ Then $$ \ker \eta_{1} \equiv [F_2,F_1,F_1]+[F_2,F_1,F_2]~~({\rm{mod}}~[R,F,F]). $$
\end{lemma}
\begin{proof}
We can see that $[F_2,F_1,F_1]+[F_2,F_1,F_2]~~~~({\rm{mod}}~[R,F,F])\subseteq \ker \eta_{1}$. Now, suppose that $x+[R,F,F] \in \ker \eta_{1}$, then by Lemma $\ref{34}$ we can write $w=v_{1}+v_{2}+v_{3}$ where $v_{1}\in [F_2,F_1,F_1]+[F_2,F_1,F_2], v_{2}\in F_1^3$ and $v_{3}\in F_2^3$. The way $\eta$ is defined, we get $v_{2}\in [R_1,F_1,F_1]$ and $v_{3}\in [R_2,F_2,F_2]$. So $x\equiv v_{1}({\rm{mod}}~[R,F,F])$, as desired.
\end{proof}


\begin{theorem}\label{37}
With the above notations we have $$[F_2,F_1,F_1]+[F_2,F_1,F_2]\equiv (L_2^{ab}\otimes L_1^{ab}\otimes L_1^{ab})\oplus(L_2^{ab}\otimes L_1^{ab}\otimes L_2^{ab})~~~({\rm{mod}}~[R,F,F]).$$
\end{theorem}

\begin{proof}
First we define $\psi_{1}:[F_2,F_1,F_1]+[F_2,F_1,F_2]/[R,F,F]\rightarrow (L_2^{ab}\otimes L_1^{ab}\otimes L_1^{ab})\oplus(L_2^{ab}\otimes L_1^{ab}\otimes L_2^{ab})$ by
$$\psi_{1}([x,y,z]+[R,F,F])=\overline{x}\otimes \overline{y}\otimes \overline{z}$$ which gives a homomorphism. Conversely, if we define $\psi_{2}:(L_2^{ab}\otimes L_1^{ab}\otimes L_1^{ab})\oplus(L_2^{ab}\otimes L_1^{ab}\otimes L_2^{ab})\rightarrow [F_2,F_1,F_1]+[F_2,F_1,F_2]/[R,F,F]$  by $$\overline{x}\otimes \overline{y}\otimes \overline{z} \mapsto [a,b,c]+[R,F,F]$$ which defines another homomorphism. Then $\psi_{1}\psi_{2}=Id=\psi_{2}\psi_{1}$. 
\end{proof}
If we know $\mathcal{M}^{(2)}(L_1)$ and $\mathcal{M}^{(2)}(L_2)$, then we can find $\mathcal{M}^{(2)}(L_1\oplus L_2)$ as follows.
\begin{theorem}\label{t6}
Let $L_1$ and $L_2$ be two Lie superalgebras, then $$\mathcal{M}^{(2)}(L_1\oplus L_2)\cong \mathcal{M}^{(2)}(L_1)\oplus\mathcal{M}^{(2)}(L_2)\oplus (L_2^{ab}\otimes L_1^{ab}\otimes L_1^{ab})\oplus(L_2^{ab}\otimes L_1^{ab}\otimes L_2^{ab}).$$
\end{theorem}

\begin{proof}
Acting Theorem $\ref{35}$ and Theorem $\ref{37}$, the desired conclusion follows.
\end{proof}


An useful result on the dimension of $2$-nilpotent multiplier of an abelian Lie superalgebra of dimension $(m\mid n)$ is given by the even and odd dimension of $2$-nilpotent multiplier.
\begin{corollary}\label{c1}
Let $A(m \mid n)$ be an abelian Lie superalgebra of dimension $(m \mid n)$. Then $$ \mathcal{M}^2(A(m|n))\cong A\bigg(\frac{1}{3}(m^3+3n^2m-m)\bigg|\frac{1}{3}(3m^2n+n^3-n)\bigg).$$
\end{corollary}
\begin{proof}
From Lemma $\ref{l12}$ it is clear that $\dim(\mathcal{M}^{(2)}(L))=\sum_{|\alpha|=3}SW(\alpha)$. Let $\alpha=(\alpha_{1},...,\alpha_{m+n})$, then $\sum_{i=1}^{m+n}\alpha_{i}=3 $. Thus, from Theorem $\ref{t11}$ at least one of $\alpha_{i}$ is odd, i.e., $\beta=0$ which implies $\sum_{|\alpha|=3}SW(\alpha)=\sum_{|\alpha|=3}W(\alpha)=\dfrac{1}{3}[(m+n)^{3}-(m+n)].$ Now the required result follows from the Corollary $\ref{c12}$.
\end{proof}

Using the above corollary, we can find the structure of $2$-nilpotent multiplier of non-capable special Heisenberg Lie superalgebras as well as odd Heisenberg Lie superalgebras \cite{Nayak2019,MC2011} . 
\begin{theorem}\label{th1}
Let $H(m, n)$ be a non-capable special Heisenberg Lie superalgebra. Then 
\begin{enumerate}
\item $\mathcal{M}^{(2)}(H(m, n)) \cong A\bigg(\frac{1}{3}(8m^3+6n^2m-2m)\bigg|\frac{1}{3}(n^3+12m^2n-n)\bigg)$ if $m+n \geq 2$
\item $\mathcal{M}^{(2)}(H(m, n)) \cong 0$ if $m=0, n=1$.
\end{enumerate}
\end{theorem}
\begin{proof}
From Lemma $\ref{l1}$ and Lemma $\ref{th5}$, we have $Z^{\ast}(H(m, n))=H(m, n)^{2}$ for $m+n \geq 2, m=0, n=1$. Now using Corollary $\ref{1200}$,
\[\mathcal{M}^{(2)}(H(m, n)) \cong \mathcal{M}^{(2)}(A(2m \mid n)). \]
Similarly, if $m=0, n=1$, then $\mathcal{M}^{(2)}(H(0,1)) \cong \mathcal{M}^{(2)}(A(0 \mid 1))$. Now the result follows from Corollary $\ref{c1}$.
\end{proof}

\begin{theorem}\label{th2}
Let $H(m)$ be a non-capable odd Heisenberg Lie superalgebra. Then $$\mathcal{M}^{(2)}(H(m)) \cong A\bigg(\frac{1}{3}(4m^3-m)\bigg|\frac{1}{3}(4m^3-m)\bigg).$$ 
\end{theorem}
\begin{proof}
The proof follows from Corollary $\ref{c1}$ and Lemma $\ref{th6a}$
\end{proof}



\begin{lemma}[see \cite{nr2016}, Theorem 2.10]\label{la}
	$\mathcal{M}^2(H(1,0))\cong \mathcal{M}^{2}(H(1))\cong A(5\mid 0).$
\end{lemma}

Now we will characterize $\mathcal{M}^2(L)$, where $L$ is a finite dimensional nilpotent Lie superalgebra whose derived subalgeba has dimension one. At first we consider nilpotent Lie superalgebras whose derived subalgeba have dimension $(1\mid 0)$.  
\begin{lemma}\label{21}
Let $L$ be a nilpotent Lie superalgebra of dimension $(k\mid l)$ with $dim~L^2=(1\mid 0)$, then $\mathcal{M}^2(L)$ is isomorphic to either $A(\dfrac{1}{3}k(k-1)(k+1)+(k-1)(l^{2}-k) \mid \dfrac{1}{3}l(l-1)(l+1)+l(k-1)^{2})$ or $A(\dfrac{1}{3}k(k-1)(k+1)+(k-1)(l^{2}-k)+1 \mid \dfrac{1}{3}l(l-1)(l+1)+l(k-1)^{2}+1)$ or $A(\dfrac{1}{3}k(k-1)(k+1)+(k-1)(l^{2}-k)+3 \mid \dfrac{1}{3}l(l-1)(l+1)+l(k-1)^{2})$.
\end{lemma}
\begin{proof}
From Lemma $\ref{th5a}$, we know that $L\cong H(m,n)\oplus A(k-2m-1|l-n)$. By Corollary $\ref{80}$, Theorem $\ref{t6}$, Theorem $\ref{c1}$, Theorem $\ref{th1}$ and Lemma $\ref{la}$, we prove the theorem by considering case by case.\\\
\textbf{Case-1:}
Suppose $m=1,n=0$. Then
\begin{align*}
\begin{split}
\mathcal{M}^2(L)&\cong \mathcal{M}^2(H(1,0))\oplus \mathcal{M}^2(A(k-3|l))\oplus \bigg(\frac{H(1,0)}{(H(1,0))^2}\otimes_{\mathbb{Z}}\frac{H(1,0)}{(H(1,0))^2}\bigg)\\ &\otimes \frac{A(k-3|l)}{(A(k-3|l))^2}\oplus \bigg(\frac{A(k-3|l)}{(A(k-3|l))^2}\otimes_{\mathbb{Z}}\frac{A(k-3|l)}{(A(k-3|l))^2}\bigg) \otimes_{\mathbb{Z}}\frac{H(1,0)}{(H(1,0))^2}\\
&\cong \mathcal{M}^2(H(1,0))\oplus \mathcal{M}^2(A(k-3|l))\oplus (A(2|0)\otimes_{\mathbb{Z}}A(2|0))\otimes_{\mathbb{Z}}A(k-3|l)\\
&\oplus (A(k-3|l)\otimes_{\mathbb{Z}}A(k-3|l))\otimes_{\mathbb{Z}}A(2|0))    \\
&\cong A(5|0)\oplus A\bigg(\frac{1}{3}((k-3)^3+3l^2(k-3)-k+3)\bigg| \frac{1}{3}(l^3+3(k-3)^2l-l)\bigg)\oplus A(4k-12|4l)\\ &\oplus A(2(k-3)^2+2l^2|4(k-3)l)\\
&\cong A\bigg(5+\frac{1}{3}(k-3)^3+l^2(k-3)-\frac{k-3}{3}+4k-12+2(k-3)^2+2l^2\bigg| \frac{1}{3}l^3+(k-3)^2l\\
&-\frac{l}{3}+4l+4l(k-3)\bigg)\\
&=A(\dfrac{1}{3}k(k-1)(k+1)+(k-1)(l^{2}-k)+3 \mid \dfrac{1}{3}l(l-1)(l+1)+l(k-1)^{2}) 
\end{split}
\end{align*}

\textbf{Case-2:}
	Suppose $m=0,n=1$, i.e., consider $L\cong H(0,1)\oplus A(k-1|l-1).$ Then
	
	\begin{align*}
	\begin{split}
	\mathcal{M}^2(L)&\cong \mathcal{M}^2(H(0,1))\oplus \mathcal{M}^2(A(k-1|l-1))\oplus \bigg(\frac{H(0,1)}{(H(0,1))^2}\otimes_{\mathbb{Z}}\frac{H(0,1)}{(H(0,1))^2}\bigg)  \otimes \frac{A(k-1|l-1)}{(A(k-1|l-1))^2}\\ &\oplus \bigg(\frac{A(k-1|l-1)}{(A(k-1|l-1))^2}\otimes_{\mathbb{Z}}\frac{A(k-1|l-1)}{(A(k-1|l-1))^2}\bigg) \otimes_{\mathbb{Z}} \frac{H(0,1)}{(H(0,1))^2}\\
	&\cong A\bigg(\frac{1}{3}((k-1)^3+3(k-1)^2(l-1)-(k-1))\bigg| \frac{1}{3}(3(k-1)^2(l-1))+(l-1)^3-(l-1)\bigg)\\
	&\oplus (A(0|1)\otimes_{\mathbb{Z}}A(0|1))\otimes_{\mathbb{Z}}A(k-1|l-1) \oplus (A(k-1|l-1)\otimes_{\mathbb{Z}} A(k-1|l-1)) \otimes_{\mathbb{Z}} A(0|1)\\
	&\cong A\bigg(\frac{1}{3}((k-1)^3+3(k-1)^2(k-1)-(k-1))\bigg| \frac{1}{3}((l-1)^3+3(k-1)^2(l-1))-(l-1))\bigg)\\&\oplus A(k-1|l-1) \oplus A(2(k-1)(l-1)|(k-1)^2+(l-1)^2)\\
	&\cong A\bigg(\frac{1}{3}(k-1)^3+(k-1)(l-1)^2-\frac{k-1}{3}+(k-1)+2(k-1)(l-1)\bigg| \frac{1}{3}(l-1)^3\\
	&+(l-1)(k-1)^2-\frac{l-1}{3}+(k-1)^2+(l-1)^2\bigg)\\
	&=A(\frac{1}{3}k(k-1)(k+1)+(k-1)(l^{2}-k)+1 \mid \frac{1}{3}l(l-1)(l+1)+l(k-1)^{2}+1).
	\end{split}
	\end{align*}
	
\textbf{Case-3:}
	Suppose $m+n\geq 2$, then $$L\cong H(m,n)\oplus A(k-2m-1|l-n).$$
	Now
	\begin{align*}
	\begin{split}
	\mathcal{M}^2(L)&\cong \mathcal{M}^2(H(m,n))\oplus \mathcal{M}^2(A(k-2m-1|l-n))\oplus \bigg(\frac{H(m,n)}{(H(m,n))^2}\otimes_{\mathbb{Z}}\frac{H(m,n)}{(H(m,n))^2}\bigg)  \otimes \frac{A(k-2m-1|l-n)}{(A(k-2m-1|l-n))^2} \\
	&\oplus \bigg(\frac{A(k-2m-1|l-n)}{(A(k-2m-1|l-n))^2}\otimes_{\mathbb{Z}}\frac{A(k-2m-1|l-n)}{(A(k-2m-1|l-n))^2}\bigg)\otimes_{\mathbb{Z}} \frac{H(m,n)}{(H(m,n))^2}\\
	&\cong A\bigg(\frac{1}{3}(8m^3+6n^2m-2m)|\frac{1}{3}(n^3+12m^2n-n)\bigg) \oplus A\bigg(\frac{1}{3}(k-2m-1)^3+(k-2m-1)(l-n)^2- \\ &\frac{1}{3}(k-2m-1)\mid \frac{1}{3}(l-n)^{3}+(k-2m-1)^{2}(l-n)-\frac{1}{3}(l-n)\bigg) \oplus A((4m^2+n^2)(k-2m-1)+4mn(l-n)|(4m^2+n^2)(l-n)+4mn(k-2m-1))\\
	&\oplus A(2m(k-2m-1)^2+2m(l-n)^2+2n(k-2m-1)(l-n)|n(k-2m-1)^2+n(l-n)^2\\&+4m(k-2m-1)(l-n))\\
	&=A(\frac{1}{3}k(k-1)(k+1)+(k-1)(l^{2}-k) \mid \frac{1}{3}l(l-1)(l+1)+l(k-1)^{2}).
	\end{split}	
	\end{align*}
	
\end{proof}

From the above, it should be noted that $\mathcal{M}^2(L)$ is isomorphic to an abelian Lie superalgebra of dimension either $\frac{1}{3}(k+l)^{3}-(k+l)^{2}+\frac{2}{3}(k+l)+3$ or $\frac{1}{3}(k+l)^{3}-(k+l)^{2}+\frac{2}{3}(k+l)+2$ or $\frac{1}{3}(k+l)^{3}-(k+l)^{2}+\frac{2}{3}(k+l)$. Using free presentation of $H_{1}$, we find the $2$-multiplier of $H_{1}$, which we give in the next theorem.



\begin{theorem}\label{70}
	$\mathcal{M}^2(H_{1})\cong A(2 \mid 2) $.
\end{theorem}
\begin{proof}
	At first we will find the free presentation of $H_1$. Let $X=\{x_{1}\}\cup\{x_{2}\}$ be a $\mathbb{Z}_2$-graded set. Let $F$ be the free Lie superalgebra over $X$. Since $H_1$ is of nilpotency class two, then $F^3 \subseteq R$, thus put 
\[R=<[x_{2},x_{2}]> + F^{3}\] 
 	
Then we can see that $0\rightarrow R\rightarrow F\rightarrow H_1 \rightarrow 0$	is a free presentation of $H_{1}$. Therefore, $$\mathcal{M}^2(H_1)\cong \frac{R\cap F^3}{[[R,F],F]} \cong \frac{R\cap F^3/ F^5}{[[R,F],F] / F^5}.$$ Using Theorem $\ref{t11}$ and Theorem $\ref{l11}$ $$\dim R\cap F^3/F^5 = \dim F^3 / F^5 =\Sigma_{|\alpha|=3}SW(\alpha) + \Sigma_{|\alpha|=4}SW(\alpha)=(1 \mid 1)+(2 \mid 2)=(3 \mid 3).$$
Now $[[R,F],F]/F^5=< [[[x_{2},x_{2}],x_{1}],x_{1}]+F^5, [[[x_{2},x_{2}],x_{1}],x_{2}]+F^5 >$. Thus, $\dim \frac{R\cap F^3}{[[R,F],F]}=(3 \mid 3)-(1 \mid 1)=(2 \mid 2)$.
\end{proof}
%
%
%
%
%

Now we are in a position to determine the $2$-nilpotent multiplier of a Lie superalgebra $L$ whose derived subalgebra has dimension $(0\mid 1).$

\begin{theorem}\label{20}
Let $L$ be a nilpotent Lie superalgebra of dimension $(k|l)$ with $dimL^2=(0|1)$. Then 
$\mathcal{M}^2(L)$ is isomorphic to $ A\big(\frac{1}{3}k(k+1)(k-1)+k(l-1)^{2}+1 \mid \frac{1}{3}l(l-1)(l+1)+(l-1)(k^{2}-l)+1 \big)$ or $ A(\frac{1}{3}k(k+1)(k-1)+k(l-1)^{2} \mid \frac{1}{3}l(l-1)(l+1)+(l-1)(k^{2}-l))$. 
\end{theorem}

\begin{proof}
	Using Lemma $\ref{th5a}$, we can decompose $L\cong H_m\oplus A(k-m|l-m-1)$. We will consider separately for different values of $m$. We conclude the following cases using Corollary $\ref{80}$, Corollary $\ref{c1}$, Theorem $\ref{th2}$ and Theorem $\ref{70}$. 
\item[\textbf{Case-1:}] 
	Suppose $m=1$, then 
	\begin{align*}
	\begin{split}
	\mathcal{M}^2(L)&\cong \mathcal{M}^2(H_1)\oplus \mathcal{M}^2(A(k-1|l-2))\oplus \bigg(\frac{H_1}{H_1^2}\otimes_{\mathbb{Z}}\frac{H_1}{H_1^2}\bigg)  \otimes \frac{A(k-1|l-2)}{(A(k-1|l-2))^2}\\ &~~~~~~~~~~~\oplus \bigg(\frac{A(k-1|l-2)}{(A(k-1|l-2))^2}\otimes_{\mathbb{Z}} \frac{A(k-1|l-2)}{(A(k-1|l-2))^2}\bigg) \otimes_{\mathbb{Z}} \frac{H_1}{H_1^2}\\
	&\cong \mathcal{M}^2(H_1)\oplus \mathcal{M}^2(A(k-1|l-2))\oplus (A(1|1)\otimes_{\mathbb{Z}} A(1|1))\otimes_{\mathbb{Z}} A(k-1|l-2))\\ &~~~~~~~~~~~~~\oplus (A(k-1|l-2))\otimes_{\mathbb{Z}} {A(k-1|l-2))}) \otimes_{\mathbb{Z}} A(1|1)\\
	& \cong A(2|2)\oplus A\bigg(\frac{1}{3}(k-1)^3+(k-1)(l-2)^2-\frac{(k-1)}{3} \bigg| \frac{1}{3}(l-2)^3+(k-1)^2(l-2)\\& ~~~~~~~~~~~~-\frac{(l-2)}{3}\bigg)\oplus A(2(k-1)+2(l-2)|2(k-1)+2(l-2))\\&~~~~~~~~~~~~\oplus A((k-1)^2+(l-2)^2+2(k-1)(l-2)|(k-1)^2+(l-2)^2+2(k-1)(l-2))\\
	&\cong A\big(\frac{1}{3}k(k+1)(k-1)+k(l-1)^{2}+1 \mid \frac{1}{3}l(l-1)(l+1)+(l-1)(k^{2}-l)+1 \big).
	\end{split}
	\end{align*}

\item[\textbf{Case-2:}]
	Suppose $m>1$, then 
	\begin{align*}
	\begin{split}
	\mathcal{M}^2(L)&\cong \mathcal{M}^2(H_m)\oplus \mathcal{M}^2(A(k-m|l-m-1))\oplus \bigg(\frac{H_m}{H_m^2}\otimes_{\mathbb{Z}}\frac{H_m}{H_m^2}\bigg)  \otimes \frac{A(k-m|l-m-1)}{(A(k-m|l-m-1))^2}\\ &~~~~~~~~~~~\oplus \bigg(\frac{A(k-m|l-m-1)}{(A(k-m|l-m-1))^2}\otimes_{\mathbb{Z}} \frac{A(k-m|l-m-1)}{(A(k-m|l-m-1))^2}\bigg) \otimes_{\mathbb{Z}} \frac{H_m}{H_m^2}\\
	&\cong \mathcal{M}^2(H_m)\oplus \mathcal{M}^2(A(k-m|l-m-1))\oplus (A(m|m)\otimes_{\mathbb{Z}} A(m|m))\otimes_{\mathbb{Z}} A(k-m|l-m-1))\\ &~~~~~~~~~~~~~\oplus (A(k-m|l-m-1))\otimes_{\mathbb{Z}} {A(k-m|l-m-1))}) \otimes_{\mathbb{Z}} A(m|m)\\
	& \cong A\bigg(\frac{1}{3}(4m^3-m)\bigg| \frac{1}{3}(4m^3-m)\bigg)\oplus A\bigg(\frac{1}{3}(k-m)^3+(k-m)(l-m-1)^2-\\&~~~~~~~~~~~~~\frac{(k-m)}{3} \bigg| \frac{1}{3}(l-m-1)^3+(k-m)^2(l-m-1) ~~~~~~~~~~~~-\frac{(l-m-1)}{3}\bigg)\oplus A(2m^2(k-m)\\&~~~~~~~~~~~~+2m^2(l-m-1)|2m^2(k-m)+2m^2(l-m-1))\oplus A(m(k-m)^2+m(l-m-1)^2\\&~~~~~~~~~~~~+2m(k-m)(l-m-1)|m(k-m)^2+m(l-m-1)^2+2m(k-m)(l-m-1))\\
&=A(\frac{1}{3}k(k+1)(k-1)+k(l-1)^{2} \mid \frac{1}{3}l(l-1)(l+1)+(l-1)(k^{2}-l)     )	
	\end{split}
	\end{align*}
		
\end{proof}

It is worth pointing out that from the Theorem $\ref{21}$ and Theorem $\ref{20}$, if $L$ is a nilpotent Lie superalgebra of dimension $(k \mid l)$ with $\dim L^2=(r\mid s)$, $r+s=1$, then $\mathcal{M}^2(L)$ is isomorphic to an abelian Lie superalgebra of dimension $\frac{1}{3}(k+l)^{3}-(k+l)^{2}+\frac{2}{3}(k+l)$ or $\frac{1}{3}(k+l)^{3}-(k+l)^{2}+\frac{2}{3}(k+l)+2$ or $\frac{1}{3}(k+l)^{3}-(k+l)^{2}+\frac{2}{3}(k+l)+3$. Now we are able to find an upper bound on the dimension of $2$-nilpotent multiplier of $L$.

\begin{theorem}\label{1121}
Let $L=L_{\bar{0}}\oplus L_{\bar{1}}$ be a nilpotent Lie superalgebra of dimension $(k\mid l)$ with $\dim L^{2}=(r \mid s),~r+s\geq 1$. Then $$\dim \mathcal{M}^2(L) \leq \frac{1}{3}(k+l-r-s)[(k+l+2r+2s-2)(k+l-r-s-1)+3(r+s-1)]+3. $$
Moreover, $\dim \mathcal{M}^2(L) \leq \frac{1}{3}(k+l)(k+l-1)(k+l-2)+3.$ Also, the equality holds in the last inequality if and only if $L \cong H(1,0)\oplus A(k-3 \mid l).$ 
\end{theorem}

\begin{proof}
We will use induction on $(r \mid s)$ to complete the proof. If $(r \mid s)=(1\mid 0)$ or $(0 \mid 1)$, then by Theorem $\ref{21}$ and Theorem $\ref{20}$ the above result holds. Now assume that $\dim L^{2}=(r \mid s)$ with $r\geq 1$ and $s\geq 1$. Then we get an one dimensional central ideal $W=W_{\bar{0}}\oplus W_{\bar{1}}$. Suppose $\dim W=(1\mid 0)$. As $W$ and $L/L^{2}$ act on each other trivially from Theorem $\ref{777}$ we have 
$$ (W\otimes L/L^{2})\otimes L/L^{2}  \cong (W\otimes_{\mathbb{Z}}\dfrac{L/L^{2}}{(L/L^{2})^2})\otimes_{\mathbb{Z}}\dfrac{L/L^{2}}{(L/L^{2})^2} .$$
Then using Theorem $\ref{777}$, $\dim \mathcal{M}^2(L)+ \dim (W \cap L^3) \leq \dim \mathcal{M}^2(L/W)+ \dim (W\otimes_{\mathbb{Z}}\dfrac{L/L^{2}}{(L/L^{2})^2})\otimes_{\mathbb{Z}}\dfrac{L/L^{2}}{(L/L^{2})^2}. $
Since $\dim L/W=(k-1 \mid l)$ and $\dim (L^{2}/W)=(r-1 \mid s)$, then from induction hypothesis
$$ \dim \mathcal{M}^2(L/W) \leq \frac{1}{3}(k+l-r-s)[(k+l+2r+2s-5)(k+l-r-s-1)+3(r+s-2)]+3.  $$
Thus, $\dim \mathcal{M}^2(L) \leq \frac{1}{3}(k+l-r-s)[(k+l+2r+2s-5)(k+l-r-s-1)+3(r+s-2)]+3 + (k+l-r-s)^{2}-1 $
$$\leq \frac{1}{3}(k+l-r-s)[(k+l+2r+2s-2)(k+l-r-s-1)+3(r+s-1)]+3.$$

\end{proof}

By considering the odd dimension zero in the above theorem, we will get the result for Lie algebra (see \cite{nr2016}, Theorem 2.14 ).
\begin{corollary}
Let $L$ be a nilpotent Lie algebra of dimension $k$ with $\dim L^{2}= r,~r\geq 1$. Then $$\dim \mathcal{M}^2(L) \leq \frac{1}{3}(k-r)[(k+2r-2)(k-r-1)+3(r-1)]+3. $$
Moreover, $\dim \mathcal{M}^2(L) \leq \frac{1}{3}(k)(k-1)(k-2)+3.$ Also, the equality holds in the last inequality if and only if $L \cong H(1)\oplus A(k-3).$ 
\end{corollary} 
An important consequence of the Theorem $\ref{1121}$ is the following corollary.
\begin{corollary}
Let $L=L_{\bar{0}}\oplus L_{\bar{1}}$ be a nilpotent Lie superalgebra of dimension $(k \mid l)$ with $k+l\geq 3$. If $\dim \mathcal{M}^2(L) \leq \frac{1}{3}(k+l)(k+l+1)(k+l-1)$, then $L \cong A(k \mid l)$.
\end{corollary}

\section{2-Capability of Lie suepralgebra}
Let $L=L_{\bar{0}}\oplus L_{\bar{1}}$ be a Lie superalgebra, then we can define the set $$Z_{2}(L)=\{x \in L \mid [x,L]\subseteq Z(L) \}. $$ Thus $Z_{2}(L)$ is a subalgebra of $L$ with $Z(L)\subseteq Z_{2}(L)$.
\begin{definition}
A Lie superalgebra $L$ is said to be $2$-capable if there exists a Lie superalgebra $H$ such that $L \cong H/Z_{2}(H)$. 
\end{definition}
A Lie superalgebra $L$ is capable \cite{RSK2019} if there exists a Lie superalgebra $H$ such that $L \cong H/Z(H)$. Suppose $L$ is $2$-capable, then $$L \cong H/Z_{2}(H)\cong \dfrac{H/Z(H)}{Z_{2}(H)/Z(H)}  \cong \dfrac{H/Z(H)}{Z(H/Z(H))}, $$ i.e., $L$ is capable. Therefore, every $2$-capable Lie superalgebra is capable. Further, we can analyze the $2$-capability of a Lie superalgebra by looking at its even part.

\begin{theorem}\label{41}
Let  $L=L_{\bar{0}}\oplus L_{\bar{1}}$ be a $2$-capable Lie superalgebra with $L_{\bar{0}}\cap Z_{2}(L)$ is non-empty.  Then $L_{0}$ is $2$-capable Lie algebra.
\end{theorem}

\begin{proof}
Since $L$ is $2$-capable, there exists a Lie superalgebra $H$ such that the map $\phi:L\rightarrow H/Z_{2}(H)$ is an isomorphism. Define the map $\psi:L_{\bar{0}}\rightarrow H_{\bar{0}}/H_{\bar{0}}\cap Z_{2}(H)$ by $\psi(x_{0})=\phi(x_{0})+H_{\bar{0}}\cap Z_{2}(H)$. Now it is easy to check that $L_{\bar{0}}\cong H_{\bar{0}}/H_{\bar{0}}\cap Z_{2}(H)$, and also $H_{\bar{0}}\cap Z_{2}(H)$ is non-empty as $L_{\bar{0}}\cap Z_{2}(L)\neq \Phi$. Therefore, $L_{0}$ is $2$-capable. 
\end{proof} 
Let $L=L_{\bar{0}}\oplus L_{\bar{1}}$ be a Lie superalgebra, then we define $Z_{2}^{\ast}(L)$ be the smallest graded ideal of $L$ such that $L/Z_{2}^{\ast}(L)$  is $2$-capable. Thus, $Z^{\ast}(L/Z_{2}^{\ast}(L))=0.$ In \cite{s2012}, $c$-capability of Lie algebra $L$ has been discussed and they have characterized $Z_{c}^{\ast}(L)$. Now we list some useful results on $2$-capability of Lie superalgebra $L$. Evidently, the proof of the following results on $Z_{2}^{\ast}(L)$ are analogous to the case of Lie algebra \cite{nr2016,RSK2019,s2012}.

\begin{proposition}\label{prop5}
Let $L$ be a Lie superalgebra with the free presentation $0\longrightarrow R \longrightarrow F \overset{\pi}{\longrightarrow} L\longrightarrow 0$, then $Z_{c}^{*}(L)=\overline{\pi}(Z_{c}(F/\gamma_{c+1}[R,F]))$.
\end{proposition}

\begin{proposition}\label{55}
 A Lie superalgebra $L=L_{\bar{0}}\oplus L_{\bar{1}}$ is $2$-capable if and only if $Z_{2}^{*}(L)=0$.
\end{proposition}

\begin{theorem}\label{56}
Let $L$ be a Lie superalgebra with a graded ideal $K$ such that $K\subseteq Z_{2}^{*}(L)$. Then the natural Lie superalgebra homomorphism $\mathcal{M}^{2}(L)\rightarrow \mathcal{M}^{2}(L/K)$ is injective.
\end{theorem}

\begin{lemma}\label{100}(see \cite{nr2016}, Theorem 3.3)
$H(m)$ is $2$-capable if and only if $m=1$.
\end{lemma}

Next theorems discusses the $2$-capability of Heisenberg Lie superalgebras of even as well as odd center.
\begin{theorem}
$H(m,n)$ is $2$-capable if and only if $m=1$, $n=0$.
\end{theorem}
\begin{proof}
As we have noticed, $2$-capability of $L$ implies capability of $L$. Thus from Lemma $\ref{th5}$ $H(m,n)$ is not $2$-capable except $m=1$, $n=0$. Since $H(1,0) \cong H(1)$. By Lemma $\ref{100}$, $H(1,0)$ is $2$-capable. This ends the proof.
\end{proof}
\begin{theorem}
$H_{m}$ is $2$-capable if and only if $m=1$.
\end{theorem}
\begin{proof}
Now from Lemma $\ref{th6a}$, $H_{m}$ is not $2$-capable for $m\geq 2$. For $H_{1}$, let us take $K=Z(H_{1})$. Then $H/Z(H_{1})\cong A(1 \mid 1)$. Thus from Theorem $\ref{70}$ and Corollary $\ref{c1}$, $\dim \mathcal{M}^{2}(H_{1})=(2\mid 2)$ and $\dim \mathcal{M}^{2}(A(1 \mid 1))=(1 \mid 1)$. Now the result follows from Proposition $\ref{55}$ and Theorem $\ref{56}$.
\end{proof}
\begin{theorem}
$A(m\mid n)$ is $2$-capable only for $m+n \geq 2$.
\end{theorem}
\begin{proof}
Let $K$ be any one dimensional subalgebra of $A(m\mid n)$ with $\dim K=(r\mid s)$, $r+s=1$. Then from Corollary $\ref{c1}$
$$\dim \mathcal{M}^2(A(m|n))= \bigg(\frac{1}{3}(m^3+3n^2m-m)\bigg|\frac{1}{3}(3m^2n+n^3-n)\bigg),$$ $$ \dim \mathcal{M}^2(A(m-r|n-s))= \bigg(\frac{1}{3}((m-r)^3+3(n-s)^2(m-r)-(m-r))\bigg|\frac{1}{3}(3((m-r)^2(n-s)+(n-s)^3-(n-s))\bigg).$$
Now from Theorem $\ref{56}$, $Z^{\ast}(A(m\mid n))=0$. Therefore, $A(m\mid n)$ is $2$-capable, by acting Proposition $\ref{55}$.
\end{proof}


\begin{thebibliography}{1}\label{reference}
	
	
	
	\bibitem{ar2013}
	M. Araskhan, 
	\newblock The dimension of the $c$-nilpotent multiplier,
	\newblock {\em J. Algebra} \textbf{386} (2013) 105-112.
	
	
	\bibitem{araskhan}
	M. Araskhan and M. R. Rismanchian,
	\newblock Dimension of the c-nilpotent multiplier of Lie algebras, 
	\newblock {\em Proc. Indian Acad. Sci. Math. Sci.} \textbf{126}(3) (2016) 353–357.
	
	
	
	\bibitem{baer}
	R. Baer,
	\newblock Representations of groups as quotient groups, I, II, and III,
	\newblock {\em Trans. Amer. Math. Soc.}  \textbf{54} (1945) 295–419. 
	
	
	\bibitem{YAM2000}
	Y. Bahturin, A. A. Mikhalev, V. M. Petrogradsky, et al.,
	\newblock Infinite dimensional Lie superalgebras,
	\newblock (De Gruyter Expositions in Mathematics \textbf{7}) Walter de Gruyter, Berlin 1992 .
	
	
	\bibitem{batten} 
	P. Batten, 
	\newblock Multipliers and covers of Lie algebras, 
	\newblock {\em PhD thesis, North Carolina State University}, 1993.
	
	\bibitem{ellis}
	G. Ellis,
	\newblock A non-abelian tensor Products of Lie algebras,
	\newblock {\em Glasg. Math. J.} \textbf{39} (1991) 101-120.
	
	
	
	
	
	\bibitem{GKL2015} 
	X. Garc\'{i}a-Mart\'{i}nez, E. Khmaladze and M. Ladra,  
	\newblock Non-abelian tensor product and homology of Lie superalgebras,
	\newblock {\em J. Algebra} \textbf{440} (2015) 464-488.
	
	\bibitem{Joh2014}
	F. Johari, P. Niroomand and M. Parvizi,
	\newblock $c$-Nilpotent Multiplier and $c$-capability of the direct sum of Lie algebras,
	\newblock {\em J. Algebra Appl.} \textbf{19}(2) (2020) 2050037.
	
	
	
	\bibitem{Kac1977}
	V. G. Kac, 
	\newblock Lie superalgebras,
	\newblock {\em Adv. Math.} \textbf{26} (1977) 8-96.
	
	
	
	
	\bibitem{Nayak2019}
	S. Nayak,
	\newblock Multipliers of nilpotent Lie superalgebras, 
	\newblock  {\em Comm. Algebra} \textbf{47}(2) (2019) 689–705.	
	
	
	\bibitem{nr2016} 
	P. Niroomand and M. Parvizi, 
	\newblock 2-capability and 2-nilpotent multiplier of finite dimensional nilpotent Lie algebras,
	\newblock {\em J. Geom. Phys.} \textbf{121} (2017) 180-185.
	
	
	\bibitem{peyman}
	P. Niroomand and M. Parvizi,
	\newblock 2-Nilpotent multipliers of a direct product of Lie algebras,
	\newblock {\em Rend. Circ. Mat. Palermo (2)} \textbf{65} (2016) 519-523.
	
	
	
	\bibitem{RSK2019}
	R. N. Padhan, S. Nayak and K. C. Pati,
	\newblock Detecting capable Lie superalgebra,
	\newblock {\em arXiv:1810.04459}.
	
	
	
	\bibitem{pet2000}
	V. M. Petrogradsky,
	\newblock On Witts's formula and invariants for free Lie superalgebras,
	\newblock {\em Formal Power Series and Algebraic Combinatorics, Moscow, Springer-Verlag, Berlin} (2000) 543-551.
	
	
	\bibitem{ris2014}
	M. R. Rismanchian and M. Araskhan,
	\newblock Some Properties of the $c$-Nilpotent Multiplier and $c$-cover of Lie algebras,
	\newblock {\em Algebra Colloq.} \textbf{21}(3) (2014) 421-426.
	
	
	\bibitem{MC2011}
	M. C. Rodriguez-Vallarte, G. Salgado and O. A. Sanchez-Valenzuela.
	\newblock Heisenberg Lie superalgebras and their invariant superorthogonal and supersympletic forms,
	\newblock {\em J. Algebra }\textbf{ 332} (2011)  71-86.
	
	
	
	
	
	
	
	\bibitem{sal2018}
	A. R. Salemkar and A. Aslizadeh,
	\newblock The nilpotent multipliers of the direct sum of Lie algebras,
	\newblock {\em J. Algebra} \textbf{496} (2018) 220-232.	   
	
	
	
	\bibitem{salemkar}
	A. R. Salemkar and B. Edalatzadeh and M. Araskhan, 
	\newblock Some inequalities for the dimension of the c-nilpotent multiplier of Lie algebras, 
	\newblock {\em J. Algebra} \textbf{322} (2009) 1575–1585.
	
	
	
	\bibitem{s2012}
	A. R. Salemkar and Z. Riyahi, 
	\newblock Some properties of the c-nilpotent multiplier of Lie algebras, 
	\newblock {\em J. Algebra} \textbf{370} (2012) 320-325.
	
	
	
	\bibitem{si1962}
	A. I. Shirshov,
	\newblock On bases of free Lie algebras,
	\newblock {\em Algebra Logika} \textbf{1}(1) (1962) 14-19.	
	
	
	
	
	
	
	
	
	
	
	
	
	
	
	
	
	
	
	
	
	
	
	
	
	
	
	
	
	
	
	
	
	
	
	
	
	
	
	
	
	
	
	
	
	
	
	
	
	
	
	
	
	
	
	
	
\end{thebibliography}
\end{document}